\documentclass[12pt]{article}
\usepackage{amssymb}
\usepackage{amsmath}
\usepackage{amsthm}
\usepackage{amsfonts}
\usepackage[ansinew]{inputenc}
\usepackage{amsthm,amsfonts}
\makeatletter \@addtoreset{equation}{section} \makeatother
\newcommand{\grad}{grad}

\newtheorem{theorem}{Theorem}[section]

\newtheorem{lemma}{Lemma}[section]
\newtheorem{proposition}{Proposition}[section]
\newtheorem{remark}{Remark}[section]
\newtheorem{assumption}{Assumption}[section]
\newtheorem{definition}{Definition}[section]

\title{Convergence of inexact descent methods for nonconvex optimization on Riemannian manifolds.}
\author{
G. C. Bento\thanks{IME, Universidade Federal de Goi\'as,
Goi\^ania, GO 74001-970, BR. (Email: {\tt glaydston@mat.ufg.br}). The author was supported in part by CNPq Grant 473756/2009-9 and PROCAD/NF. {\bf Corresponding author. Tel.: +55 62 3521 1418; fax: +55 62 3521 1180.}}
\and
 J. X. da Cruz Neto
\thanks{Centro de Ciências da Natureza, DM, Universidade Federal do Piau\'i,
Terezina, PI 64049-550, BR (Email: {\tt jcruzneto@uol.com.br}).
The author was supported in part by PRONEX-Optimization(FAPERJ/CNPq)DM.}\and
P. R. Oliveira \thanks{COPPE-Sistemas, Universidade Federal do Rio de Janeiro,
Rio de Janeiro, RJ 21945-970, BR (Email: {\tt poliveir@cos.ufrj.br}).
This author was supported in part by CNPq.}
}
\begin{document}
\maketitle

\begin{abstract}
In this paper we present an abstract convergence analysis of inexact descent methods in Riemannian context for functions satisfying Kurdyka-Lojasiewicz inequality. In particular, without any restrictive assumption about the sign of the sectional curvature of the manifold, we obtain full convergence of a bounded sequence generated by the proximal point method, in the case that the objective function is nonsmooth and nonconvex, and the subproblems are determined by a  quasi distance which does not necessarily coincide with the Riemannian distance.  Moreover, if the objective function is $C^1$ with $L$-Lipschitz gradient, not necessarily convex, but satisfying Kurdyka-Lojasiewicz inequality, full convergence of a bounded sequence generated by the steepest descent method is obtained. 
\end{abstract}


\section{Introduction} 
Consider the following minimization problem
\begin{eqnarray}\label{po:conv777}
\begin{array}{clc}
   & \min  f(x) \\
   & \textnormal{s.t.}\,\,\, x\in M,\\
\end{array}
\end{eqnarray}
where $M$ is a complete Riemannian manifold and $f:M\to\mathbb{R}\cup\{+\infty\}$ is a proper lower semicontinuous function bounded from below. The exact proximal point method to solve optimization problems of the
form \eqref{po:conv777} generates, for a starting point $x^0\in M$, a sequence $\{x^k\}\subset M$ as follows:
\begin{equation}\label{eq:prox1}
x^{k+1}\in \mbox{argmin}_{y\in M}\left\{f(y)+\lambda_k d^2(y,x^k)\right\},
\end{equation}
where $\{\lambda_k\}$ is a sequence of positive numbers and $d$ is the Riemannian distance (see Section~\ref{sec2} for a definition). This method was first considered in this context by Ferreira and Oliveira \cite{FO2000}, in the particular case that $M$ is a Hadamard manifold (see Section~\ref{sec2} for a definition), $\mbox{dom}f=M$ and $f$ is convex. They proved that, for each $k\in \mathbb{N}$, the function $f(.)+d^2(.,x^k):M\to\mathbb{R}$ is 1-coercive and, consequently, that the sequence $\{x^k\}$ is well-defined, with $x^{k+1}$ being uniquely determined. Moreover, supposing $\sum_{k=0}^{+\infty}1/\lambda_k=+\infty$ and that $f$ has a minimizer, the authors proved convergence of the sequence $\{f(x^k)\}$ to the minimum value and convergence of the sequence $\{x^k\}$ to a minimizer point. Li et al.~\cite{Chong Li2009} extended this method for finding singularity of a multivalued vector field and proved that the generated sequence is well-defined and converges to a singularity of a maximal monotone vector field, whenever it exists.

In the last three decades, several authors have proposed the generalized proximal point method for certain nonconvex minimization problems. As far as we know the first direct generalization, in the case that $M$ is a Hilbert space, has been performed by Fukushima and Mine~\cite{Fukushima1981}. See Kaplan and Tichatschke~\cite{Kaplan1998} for a review. For the problem of finding singularities of multivalued operators, that situation is similar to the case where there is no monotonicity (e.g. where the operator is hypomonotone), see e.g., Spingarn and Jonathan \cite{Spingarn1982}, Pennanen \cite{pennanen2002}, Iusem et al. \cite{pennanen2003}, Combettes and Pennanen \cite{Combettes2004}, Garciga and Iusem \cite{Iusem2007}. In the Riemannian context, Papa Quiroz and Oliveira \cite{PO2009} considered the proximal point method for quasiconvex function (not necessarily convex) and proved full convergence of the sequence $\{x^k\}$ to a minimizer point with $M$ being a Hadamard manifold. Bento et al. \cite{BFO2010} considered the proximal point method for $C^1$-lower type functions and obtained local convergence of the generated sequence to a minimizer, also in the case that $M$ is a Hadamard manifold. 

So far, in the convergence analysis of the exact proximal point method for solving either convex minimization problems (resp. find singularities of monotone vector field) how nonconvex minimization problems, it was necessary to consider Hadamard type manifolds. This is because the convergence analysis is based on Fejér convergence to the minimizers set of $f$ (resp. to the singularities set of the vector field), and these manifolds, apart from having the same topology and structure differentiable from Euclidean space, also have geometric properties satisfactory to the characterization of Fejér convergence of the sequence. 

In the Riemannian context we raise the following issue:
\vspace{0.5cm}

{\bf Problem 1.} {\it Would it be possible to obtain the convergence of the whole sequence generated by the method \eqref{eq:prox1} for $f$ not necessarily convex or quasiconvexa, and $M$ not necessarily Hadamard?}

Considering again Problem~\eqref{po:conv777} now with $f$ continuously differentiable and $\mbox{dom}f=M$, given $x^0\in M$ the classic steepest descent method generates a sequence $\{x^k\}$ given by
\begin{equation}\label{eq:sdm1}
x^{k+1}=\exp_{x^k}(-t_k\mbox{grad} f(x^k)),
\end{equation}
where exp is the exponential map and $t_k$ is some positive stepsize. 
As far as we know, this method was first studied by Luemberger~\cite{Luenberger1972} and later by Gabay~\cite{Gabay1982} both in the particular case where $M$ is the inverse image of regular value. Udriste~\cite{U94}, Smith~\cite{S94}, Rapcs\'ak~\cite{RAP97} also studied whose method in the case that $M$ is a any complete Riemannian manifold and partial convergence results were obtained. For the convex case, the full convergence, using Armijo's rule and 'fixed step' was obtained by da Cruz Neto et al. \cite{XLO1998}, in the particular case where $M$ has nonnegative sectional curvature. Regarding  the same restrictive assumption on the manifold $M$, Papa Quiroz et al. \cite{PP08} proved a full convergence result using a generalized Armijo's rule for the quasiconvex case. Note that the results of convergence presented in \cite{XLO1998} (resp. \cite{PP08}) for the steepest descent method for solving convex minimization problems depend, besides the assumption of convexity (resp. quasi-convexity) on function $f$, of the sign of the sectional curvature of M. This is because the convergence analysis is based on quasi-Fejér convergence to the minimizers set of $f$ and these manifolds have geometric properties favorable to the characterization of the quasi-Fejér convergence of the sequence. 

In the same context of Problem 1, we raise the following issue:
\vspace{0.5cm}

{\bf Problem 2.} {\it Would it be possible to obtain the convergence of the whole sequence generated by the method \eqref{eq:sdm1} for $f$ not necessarily convex or quasiconvexa, and without restrictive assumption on the sign of the sectional curvature of the manifold $M$?}

This paper has the same spirit of the work of Alvarez et al. \cite{Alvarez2008} in which the authors proved a unified result
for the existence and local uniqueness of the solution, and for the local convergence
of a Riemannian version of Newton's method. Here we are interested in providing an unified framework for the convergence analysis of classical descent methods which, in particular, includes the methods  \eqref{eq:prox1} and \eqref{eq:sdm1}, and answer Problems 1 and 2. To achieve this goal, we assume, as main assumption, that the objective function satisfies a well-known property as Kurdyka-Lojasiewicz inequality. This inequality was introduced by Kurdyka \cite{Kurdyka1998}, for differentiable functions definable in an o-minimal structure defined in $\mathbb{R}^n$, through the following result:

{\it Given $U\subset\mathbb{R}^n$ a bounded open set and $g:U\to\mathbb{R}_+$ a differentiable function definable on a o-minimal structure, there exists $c,\eta>0$ and a strictly increasing positive function definable $\varphi:\mathbb{R}_+\to\mathbb{R}$ of class $C^1$, such that
\begin{equation}\label{ineq:Kurdyka1}
\|\nabla \left(\varphi \circ g\right)(x)\|\geq c, \qquad x\in U\cap g^{-1}(0,\eta).
\end{equation}}
Note that taking $\varphi(t)=t^{1-\alpha}$, $\alpha\in [0,1)$, the inequality \eqref{ineq:Kurdyka1} yields
\begin{equation}\label{ineq:Kurdyka2}
\|\nabla g(x)\|\geq c|g(x)|^{\alpha}, 
\end{equation}
where $c=1/(1-\alpha)$, which is known as Lojasiewicz inequality, see \cite{Lojasiewicz1963}. Absil et al. \cite{Absil2005} established an abstract result of convergence for sequences (satisfying a strong descent condition) in the case where the objective function is analytic, defined on $\mathbb{R}^n$, which satisfies \eqref{ineq:Kurdyka2}. For extensions of Kurdyka-Lojasiewicz inequality to subanalytic nonsmooth functions (defined in Euclidean spaces) see, for example, Bolte et al. \cite{Bolte2006},  Attouch and Bolte~\cite{Attouch2009}.  A more general extension, yet in the context Euclidean, was developed by Bolte et al. \cite{Bolte2007} mainly for Clarke's subdifferentiable of a lower semicontinuous function definable in an o-minimal structure. Also in the Euclidean context, Attouch et al.~\cite{Attouch2010} presented a general convergence result for Inexact gradient methods in the case that the objective function satisfies \eqref{ineq:Kurdyka1}. Lageman~\cite{Lageman2007} extended the Kurdyka-Lojasiewicz inequality~\eqref{ineq:Kurdyka1} for analytic manifolds and differentiable $\cal{C}$-functions in an analytic-geometric category (satisfying a certain descent condition, namely, angle and Wolfe-Powell conditions) and established an abstract result of convergence of descent method, see \cite[Theorem 2.1.22]{Lageman2007}. In particular, Lageman observed that the answer to {\bf Problem 2} is positive (see Example of Theorem 2.1.22, page 96 of \cite{Lageman2007}).  It is important to note that Kurdyka et al.~\cite{Kurdyka2000} had already established an extension of the inequality~\eqref{ineq:Kurdyka2} for analytic manifolds and analytic functions to solve R. Thom's conjecture. 

The paper is organized as follows. In Section~\ref{sec2} we recorded some basic definitions and results
of the theory of Riemannian manifolds. In Section~\ref{sec3} we present elements of nonsmooth analysis on manifold. In Section~\ref{sec4} we present the Kurdyka-Lojasiewicz inequality in the Riemannian context  and recall some basic notions on o-minimal structures on $(\mathbb{R} ,+,\cdot )$ and analytic-geometric categories. In Section~\ref{sec5} we present an abstract converge analysis of inexact descent methods for functions satisfying Kurdyka-Lojasiewicz inequality. In Section~\ref{sec6} we recall the exact proximal point method in the Riemannian context an inexact version of it in that context. Finally, in Section~\ref{sec7} we recall the gradient method in the Riemannian context and we extend some the convergence results for KL functions.

\section{Preliminary of Riemanian Geometry}\label{sec2}

In this section, we introduce some fundamental properties and notations of Riemannian manifold. These basics facts can be found in any introductory book to Riemannian geometry, such as in do Carmo \cite{MP92} or Sakai \cite{S96}.

Let $M$ be a $n$-dimentional connected manifold. We denote by $T_xM$ the $n$-dimentional {\it tangent space} of $M$ at $x$, by $TM=\cup_{x\in M}T_xM$ {\itshape{tangent bundle}} of $M$ and by ${\cal X}(M)$ the space of smooth vector fields over $M$. When $M$ is endowed with a Riemannian metric $\langle .\,,\,. \rangle$, with the corresponding norm denoted by $\| . \|$, then $M$ is a Riemannian manifold. Remember that the metric can be used to define the length of piecewise smooth curves $\gamma:[a,b]\rightarrow M$ joining $x$ to $y$, i.e., such that $\gamma(a)=x$ and $\gamma(b)=y$, by:
\[
l(\gamma)=\int_a^b\|\gamma^{\prime}(t)\|dt,
\]
and, moreover, by minimizing this length functional over the set of all such curves, we obtain a Riemannian distance $d(x,y)$ inducing the original topology on $M$. We denote by $B(x,\epsilon)$ the Riemannian ball on $M$ with center $x$ and radius $\epsilon>0$. The metric induces a map $f\mapsto\mbox{grad} f\in{\cal X}(M)$ which, for each function smooth over $M$, associates its gradient via the rule $\langle\mbox{grad} f,X\rangle=d f(X),\ X\in{\cal X}(M)$. Let $\nabla$ be the Levi-Civita connection associated with $(M,{\langle} \,,\, {\rangle})$. In each point $x\in M$, we have a linear map $A_X(x) \colon T_xM \to T_xM$ defined by:
\begin{equation}
A_X(x)v=\nabla_vX.
\end{equation}
If $X=\mbox{grad} f$, where $f\colon M \to\mathbb{R}$ is a twice differentiable function, then $A_X(x)$ is the {\it Hessian\/} of $f$ at $x$ and is denoted by $\text{Hess}f$. A vector field $V$ along $\gamma$ is said to be {\it parallel} if $\nabla_{\gamma^{\prime}} V=0$. If $\gamma^{\prime}$ itself is parallel we say that $\gamma$ is a {\it geodesic}. Given that the geodesic equation $\nabla_{\ \gamma^{\prime}} \gamma^{\prime}=0$ is a second-order nonlinear ordinary differential equation, we conclude that the geodesic $\gamma=\gamma _{v}(.,x)$ is determined by its position $x$ and velocity $v$ at $x$. It is easy to verify that $\|\gamma ^{\prime}\|$ is constant. We say that $ \gamma $ is {\it normalized} if $\| \gamma ^{\prime}\|=1$. The restriction of a geodesic to a closed bounded interval is called a {\it geodesic segment}. A geodesic segment joining $x$ to $y$ in $ M$ is said to be {\it minimal} if its length equals $d(x,y)$ and the geodesic in question is said to be a {\it minimizing geodesic}. If $\gamma$ is a geodesic joining points $x$ and $y$ in $ M$ then, for each $t\in [a,b]$, $\nabla$ induces a linear isometry, relative to ${ \langle}\, ,\, {\rangle}$, $P_{\gamma(a)\gamma(t)}:T_{\gamma(a)}M\to T_{\gamma(t)}M$, the so-called {\it parallel transport} along $\gamma$ from $\gamma(a)$ to $\gamma(t)$. The inverse map of $P_{\gamma(a)\gamma(t)}$ is denoted by $P_{\gamma(a)\gamma(t)}^{-1}:T_{\gamma(t)} M \to T_{\gamma(a)}M$. In the particular case of $\gamma$ to be the unique geodesic segment joining $x$ and $y$, then the parallel transport along $\gamma$ from $x$ to $y$ is denoted by $P_{xy}:T_{x}M\to T_{y}M$.

A Riemannian manifold is {\it complete} if the geodesics are defined for any values of $t$. Hopf-Rinow's theorem (see, for example, Theorem 2.8, page 146 of \cite{MP92} or Theorem 1.1, page 84 of \cite{S96}) asserts that if this is the case then any pair of points, say $x$ and $y$, in $M$ can be joined by a (not necessarily unique) minimal geodesic segment. Moreover, $( M, d)$ is a complete metric space so that bounded and closed subsets are compact. From the completeness of the Riemannian manifold $M$, the {\it exponential map} $\exp_{x}:T_{x}  M \to M $ is defined by $\exp_{x}v\,=\, \gamma _{v}(1,x)$, for each $x\in M$. We denote by $B_{\epsilon}(0_{x})$ the ball in the tangent space $T_xM$ with center $0_x$ and radius $\epsilon>0$. Since the $D\exp_{x}(0_{x})$ is the identity, then by inverse mapping theorem there exists an $\epsilon >0$ such that $exp_{x}\mid_{B_{\epsilon}(0_{x})}$ is a diffeomorphism onto an open set $\mathcal{U}_x$ in $M$ containing $x$. We call the open set $\mathcal{U}_x$ a normal neighbourhood of $x$. It can be shown that, for each $v\in B_{\epsilon}(0_{x})$,
\begin{equation}\label{normal}
y=exp_{x}(v)\in \mathcal{U}_x \Rightarrow d(x,y)=\|v\|.
\end{equation}

We denote by $R$ {\it the curvature tensor \/} defined by
$R(X,Y)=\nabla_{X}\nabla_{Y}Z- \nabla_{Y}\nabla_{X}Z-\nabla_{[Y,X]}Z$, with $X,Y,Z\in{\cal X}(M)$, where $[X,Y]=YX-XY$. Moreover, the {\it sectional curvature \/} as related 
to $X$ and $Y$ is given by $K(X,Y)=\langle R(X,Y)Y , X\rangle /(||X||^{2}||X||^{2}-\langle X\,,\,Y\rangle ^{2})$, where $||X||=\langle X,X\rangle ^{2}$. If $K(X,Y)\leqslant 0$ for all $X$ and $Y$, then $M$ is called  a {\it Riemannian manifold of nonpositive curvature \/} and we use the short notation $K\leqslant 0$.

A complete, simply connected Riemannian manifold of nonpositive sectional curvature is called a {\it{Hadamard manifold}}. It is known that if $M$ is a Hadamard manifold, then $M$ has the same topology and differential structure of the Euclidean space $\mathbb{R}^n$, see for example \cite[Lemma 3.2, page 149]{MP92} or \cite[Theorem 4.1, page 221]{S96}. Furthermore, some similar geometrical properties are known of the Euclidean space $\mathbb{R}^n$, such as that where, given two points, there
exists an unique geodesic segment that joins them. 

{\it In this paper $M$ denote a complete n-dimensional Riemannian manifold}.

\section{Nonsmooth analysis on manifold} \label{sec3}

In this section we present elements of nonsmooth analysis on manifold, which can be found, for example, in YU et al \cite{Zhu2007}. 

Let $f:M\to \mathbb{R} \cup \{+\infty\}$ be a real extended-valued function and denoted by 
\[
\mbox{dom}f:=\{x\in M : f(x) < +\infty\}
\] 
its domain. We recall that $f$ is said to be proper when $\mbox{dom}f\neq\emptyset$.

\begin{definition}\label{eq:subfrechet}
Let $f$ be a lower semicontinuous function. The Fréchet-subdifferential of $f$ at $x\in M$ is defined by
\[
\hat{\partial}f(x)=
\left\{\begin{array}{cccc}
\{dh_x: h\in C^1(M)\;\mbox{and}\; f-h\; \mbox{attains a local minimum at}\; x\},\quad \mbox{if}\quad x \in \mbox{dom} f\\
\hspace{0.5cm}\emptyset,\quad\qquad \qquad\qquad\qquad  \qquad \qquad\qquad\qquad  \qquad \qquad\qquad\qquad\mbox{if}\quad x \notin \mbox{dom} f,
\end{array}\right.
\]
where $dh_x\in (T_xM)^\ast$ is given by $dh_x(v) = \langle \mbox{grad\;} h(x),v\rangle$, $v \in T_xM$. 
\end{definition}
Note that if $f$ is differentiable at $x$, then $\hat{\partial} f (x)=\{\grad f(x)\}$.
\begin{definition}\label{eq:sublimiting}
Let $f$ be a lower semicontinuous function. The (limiting) subdifferential of $f$ at $x\in M$ is defined by
\[
\partial f(x):=\{v\in T_xM: \exists (x^n, v^n)\in\mbox{Graph}(\hat{\partial}f)\; \mbox{with}\; (x^n,v^n)\to (x,v),\; f(x^n)\to f(x)\},
\]
where $\mbox{Graph}(\hat{\partial}f):=\{(y,u)\in TM: u\in\hat{\partial}f(y)\}$.
\end{definition}

It follows directly from Definitions \ref{eq:subfrechet} and \ref{eq:sublimiting} that $\hat{\partial}f(x)\subset \partial f(x)$. Note also that, $\hat{\partial}f(x)$ may be empty, but if $f$ attains a local minimum at $x$, then $0\in \hat{\partial}f(x)$. A necessary (but not sufficient) condition for $x\in M$ to be a minimizer of $f$ is
\[
0\in\partial f(x).
\] 
A point $x\in M$ satisfying the above inclusion is called limiting-critical or simply critical.
\begin{proposition}\label{regracadeiaZhu}
Let $f :M\to\mathbb{R}\cup\{+\infty\}$ be a lower semicontinuous function. Suppose that
$(U,\phi)$ is a local coordinate neighborhood and $x\in U$. Then,
\[
\partial f(x)=(\phi^*_x)\partial (f\circ \phi^{-1})(\phi(x)),
\]
where $\phi^*_x$ denote the Fréchet derivative adjunct of the function $\phi$.
\end{proposition}
\begin{proof}
See \cite[Corollary 4.2]{Zhu2007}.
\end{proof}
\section{Kurdyka-Lojasiewicz inequality on Riemannian manifolds}\label{sec4}
In this section we present the Kurdyka-Lojasiewicz inequality in the Riemannian context and recall some basic notions on o-minimal structures on $(\mathbb{R} ,+,\cdot )$ and analytic-geometric categories. Our main interest here is to observe that the Kurdyka-Lojasiewicz inequality, in Riemannian context, holds for lower semicontinuous functions, not necessarily differentiable. The differentiable case was presented by Lageman~\cite[Corollary 1.1.25]{Lageman2007}. It is important to note that Kurdyka et al.~\cite{Kurdyka2000} had already established such inequality for analytic manifolds and analytic functions. For a detailed discussion on o-minimal structures and analytic geometric categories see, for example, Dries and Miller~\cite{Dries1996}, and references therein.

Let $f:M\to \mathbb{R} \cup \{+\infty\}$ be a proper lower semicontinuous function and we consider the following sets:
\begin{itemize}
\item $dist(0,\partial f(x)):=inf\{\| v \|: v \in \partial f(x)\}$,
\item $[\eta_1 <f<\eta_2]:=\{x\in M: \eta_1 < f(x) < \eta_2\},\quad -\infty<\eta_1<\eta_2<+\infty$.
\end{itemize}
\begin{definition}\label{dkl}
The function $f$ is said to have the Kurdyka-Lojasiewicz property at $\bar{x}\in\mbox{dom}\; \partial f$ if there exists $\eta \in (0,+\infty]$, a neighborhood $U$ of $\bar{x}$ and a continuous concave function $\varphi:[0,\eta)\rightarrow \mathbb{R}_+$ such that:
\begin{itemize}
\item [(i)] $\varphi(0)=0$, $\varphi\in C^1(0,\eta)$ and, for all $s\in(0,\eta)$, $\varphi '(s)>0;$
\item [(ii)] for all $x\in U\cap [f(\bar{x})<f<f(\bar{x})+\eta]$, the Kurdyka-Lojasiewicz inequality holds
\begin{equation}\label{eq:kur100}
\varphi '(f(x)-f(\bar{x}))dist(0, \partial f(x))\geq 1.
\end{equation}
\end{itemize}
We call $f$ a KL function, if it satisfies the Kurdyka-Lojasiewicz inequality at each point of \mbox{dom}$\partial f$.
\end{definition}

Next we show that if $\bar{x}$ is a noncritical point of a lower semicontinuous function then the Kurdyka-Lojasiewicz inequality holds in $\bar{x}$.
\begin{lemma}\label{Lem3.2}
Let $f:M\to\mathbb{R}\cup\{+\infty\}$ be a proper lower semicontinuous function and $\bar{x}\in \mbox{dom}\partial f$ such that $0\notin \partial f(\bar{x})$. Then, the Kurdyka-Lojasiewicz inequality holds in $\bar{x}$.
\end{lemma}
\begin{proof} 
Since $\bar{x}$ is a noncritical point of $f$ and $\partial f(\bar{x})$ is a closed set, we have that 
\[
\delta:=dist(0,\partial f(\bar{x}))>0.
\]
Take $\varphi(t):=t/\delta$, $U:=B(\bar{x},\delta/2)$, $\eta:=\delta/2$ and note that, for each $x\in \mbox{dom}\partial f$,
\begin{equation}\label{eq:ineq100}
\varphi'(f(x)-f(\bar{x}))dist(0,\partial f(x))=dist(0,\partial f(x))/\delta.
\end{equation}
Now, for each $x\in U\cap [f(\bar{x})-\eta<f<f(\bar{x})+\eta]$ arbitrary, note that
\[
d(x,\bar{x})+|f(x)-f(\bar{x})|<\delta.
\]
We state that, for each $x$ satisfying the last inequality, it holds
\begin{equation}\label{eq:ineq101}
dist(0,\partial f(x))\geq \delta.
\end{equation}
Let us suppose, by contradiction, that this does not holds. Then, there exist sequences $\{(y^k,v^k)\}\subset\mbox{Graph}\partial f$ and $\{\delta^k\}\subset\mathbb{R}_{++}$ such that 
\[
d(y^k,\bar{x})+|f(y^k)-f(\bar{x})|<\delta^k,\quad \mbox{and}\quad \|v^k\|\leq \delta^k,
\]
with $\{\delta^k\}$ converging to zero. Thus, using that $\{(y^k,v^k)\}$ and $\{f(y^k)\}$ converge to $(\bar{x},0)$ and $f(\bar{x})$ respectively, and $\partial f$ is a closed mapping, it follows that $\bar{x}$ is a critical point of $f$, which proves the statement. Therefore, the result of the lemma follows by combining \eqref{eq:ineq100} with \eqref{eq:ineq101}.
\end{proof}

It is known that a $C^2$-function $f:M\to\mathbb{R}$ is a Morse function if each critical point $\bar{x}$ of $f$ is nondegenerate, i.e, if \mbox{Hess}$f(\bar{x})$ has all its eigenvalues different of zero. From the inverse function theorem it follows that the critical points of a Morse function are isolated. It is also known, see \cite[Theorem 1.2, page 147 ]{Hirsh76}, that Morse functions form a dense and open set in the space of $C^2$-function, more precisely 
\begin{theorem}
Let $M$ be a manifold and denote by $C^r(M,\mathbb{R})$, the set of all the $C^r$-functions $g:M\to\mathbb{R}$. The collection of all the Morse functions form a dense and open set in $C^r(M,\mathbb{R})$, $2\leq r\leq +\infty$.
\end{theorem}
Let $f:M\to\mathbb{R}$ be a Morse function and $\bar{x}\in M$ be a critical point of $f$, and take $U=B(\bar{x},\delta)\subset\mathcal{U}_{\bar{x}}$ such that it does not contain another critical point. Using the Taylor formula for $f$ and \mbox{grad}$f$ and taking into account \eqref{normal}, we obtain, for $x\in U$ 
\[
f(x)-f(\bar{x})=\frac{1}{2}\langle \mbox{Hess}\,f(\bar{x})\exp^{-1}_{\bar{x}}x,\exp^{-1}_{\bar{x}}x\rangle+o(d^2(x,\bar{x})),
\]
\[
\mbox{grad}f(x)=\mbox{Hess}\,f(\bar{x})\exp^{-1}_{\bar{x}}x+o(d(x,\bar{x})),
\]
Reducing the size of the radius $\delta$, if necessary, we can ensure the existence of positive constants $\delta_1,\delta_2$ such that
\[
|f(x)-f(\bar{x})|\leq\delta_{1}d^2(x,\bar{x}) \quad \mbox{and}\quad \delta_{2}d(x,\bar{x})\leq\|\mbox{grad}f(x)\|.
\]
From the last two inequalities, it is easy to verify that \eqref{eq:kur100} holds with $\varphi(s)=2\sqrt{\delta_{1}s}/\delta_{2}$, $U=B(\bar{x},\delta)$ and $\eta=\delta$. Therefore, it follows from Lemma~\ref{Lem3.2} that the Morse functions are KL functions. 
\begin{remark}
It should be pointed that the last examples, amongst other things, also have appeared in \cite{Attouch2010-2} in the Euclidean context. For examples illustrating failure of this property see, for instance, \cite{Absil2005, Bolte2006, Bolte2007}.
\end{remark}

Next we recall some definitions which refer to o-minimal structures on $(\mathbb{R} ,+, \cdot )$, following the notations of Bolte et al.\cite{Bolte2007}.
\begin{definition}\label{def5}
Let $\mathcal{O} =\{\mathcal{O}_n\}_{n\in\mathbb{N}}$ be a sequence such that each $\mathcal{O}_n$ is a collection of subsets of $\mathbb{R}^n$. $\mathcal{O}$ is said to be an o-minimal structure on the real field $(\mathbb{R},+,\cdot)$ if, for each $n\in\mathbb{N}$:
\begin{itemize}
\item[(i)] $\mathcal{O}_n$ is a Boolean Algebra;
\item[(ii)] If $A\in \mathcal{O}_n$, then $A\times\mathbb{R}\in \mathcal{O}_{n+1}$ and $\mathbb{R}\times A \in \mathcal{O}_{n+1}$;
\item[(iii)] If $A \in\mathcal{O} _{n+1}$, then $\pi_{n}(A)\in\mathcal{O}_n$, where $\pi _n : \mathbb{R}^{n+1}\to\mathbb{R}^n$ is the projection on the first $n$ coordinates;
\item[(iv)] $\mathcal{O}_n$ contains the family of algebraic subsets of $\mathbb{R}^n$; 
\item[(v)] $\mathcal{O}_1$ consists of all finite unions of points and open intervals.
\end{itemize}
\end{definition}

The elements of $\mathcal{O}$ are said to be {\it definable} in $\mathcal{O}$. A function $f:\mathbb{R}^n\to \mathbb{R}$ is said to be {\it definable} in $\mathcal{O}$ if its graph belongs to $\mathcal{O}_{n+1}$. Moreover, according to Coste~\cite{Coste2000} $f:\mathbb{R}^n\to\mathbb{R}\cup\{+\infty\}$ is said to be definable in $\mathcal{O}$ if the inverse images of $f^{-1}(+\infty)$ is definable subset of $\mathbb{R}^n$ and the restriction of $f$ to $f^{-1}(\mathbb{R})$ is a definable function with values in $\mathbb{R}$.
It is worth noting that an o-minimal structure on the real field $(\mathbb{R},+,\cdot)$ is a generalization of a semialgebraic set on $\mathbb{R}^n$, i.e., a set that can be written as a finite union of sets of the form
\[
\{x\in\mathbb{R}^n: p_i(x)=0,\; q_i(x)<0,\; i=1,\ldots,r\},
\] 
with $p_i, q_i$, $i=1,\ldots, r$, being real polynomial functions. Bolte et al. \cite{Bolte2007}, presented a nonsmooth extension of the Kurdyka-Lojasiewicz inequality for definable functions, but in the case that the function $\varphi$, which appears in Definition~\ref{dkl}, is not necessarily concave. Attouch et al. \cite{Attouch2010-2}, reconsidered the said extension by noting that $\varphi$ may be taken concave. For an extensive list of examples of definable sets and functions on an o-minimal structure and properties see, for example, \cite{Dries1996, Attouch2010-2}, and references therein. We limit ourselves to presenting just the material needed for our purposes.

The first elementary class of examples of definable sets is given by the semi-algebraic sets, which we denote by $\mathbb{R}_{alg}$. An other class of examples, which we denoted by $\mathbb{R}_{an}$, is given by restricted analytic functions, i.e., the smallest structure
containing the graphs of all $f\mid_{[0,1]^n}$ analytic, where $f:\mathbb{R}^n\to\mathbb{R}$ is an arbitrary function that vanishes identically off $[0,1]^n$.

Fulfilling the same role as the semi-algebraic sets on X, on analytic manifolds we have the semi-analytic and sub-analytic sets which we define below, see Bierstone and Milman\cite{Milman1988}, Dries \cite{Dries2000}:

{\it A subset of an analytic manifold is said to be semi-analytic if it is locally described by a finite number of analytic equations and inequalities, while the sub-analytic ones are local projections of relatively compact semi-analytic sets.}

A generalization of semi-analytic and sub-analytic sets, analogous to what was given to semi-algebraic sets in terms of the o-minimal structure, leads us to the analytic-geometric categories which we define below:
\begin{definition}\label{def6}
An analytic-geometric category $\mathcal{C}$ assigns to each real analytic
manifold $M$ a collection of sets $\mathcal{C}(M)$ such that for all real analytic manifolds
$M,\ N$ the following conditions hold:
\begin{itemize}
\item[(i)] $\mathcal{C}(M)$ is a Boolean Algebra of subsets of $M$, with $M\in\mathcal{C}(M)$;
\item[(ii)] If $A\in \mathcal{C}(M)$, then $A\times \mathbb{R} \in \mathcal{C}(A\times \mathbb{R})$;
\item[(iii)] If $f: M \rightarrow N$ is a proper analytic map and $A\in \mathcal{C}(M)$, then $f(A) \in  \mathcal{C}(N)$;
\item[(iv)] If $A\subset M$ and $\{U_i\mid i\in \Lambda \}$ is an open covering of $M$, then $A\in \mathcal{C}(M)$ if and only if $A\cap U_i \in \mathcal{C}(U_i)$, for all $i \in \Lambda$;
\item[(v)] Every bounded set $A \in \mathcal{C}(\mathbb{R})$ has finite boundary, i.e. the topological boundary, $\partial A$, consists of a finite number of points.
\end{itemize}
\end{definition}
The elements of $\mathcal{C}(M)$ are called $\mathcal{C}$-sets. If the graph of a continuous function $f:A \rightarrow B$ with $ A \in \mathcal{C}(M), B \in \mathcal{C}(N)$ is contained in $\mathcal{C}(M \times N)$, then $f$ is called a $\mathcal{C}$-function. All subanalytic subsets and continuous subanalytic map of a manifold are $\mathcal{C}$-sets and $\mathcal{C}$-functions  respectively, in that manifold. We denoted this collection by $\mathcal{C}_{an}$ which represents the 'smallest' analytic-geometric category.

The next theorem provides a biunivocal correspondence between o-minimal structures containing $\mathbb{R}_{an}$ and an analytic-geometric category.
\begin{theorem}\label{teo3}
For any analytic-geometric category $\mathcal{C}$ there is an o-minimal structure $\mathcal{O(C)}$ and for any o-minimal structure $\mathcal{O}$ on $\mathbb{R}_{an}$ there is an analytic geometric category $\mathcal{C(O)}$, such that
\begin{itemize}
\item[(i)] $A\in \mathcal{C(O)}$ if for all $x\in M$ exists an analytic chart $\phi: U \rightarrow \mathbb{R}^n$ , $x\in U$, which maps $A\cap U$ onto a set definable in $\mathcal{O}$.
\item[(ii)] $A \in \mathcal{O(C)}$ if it is mapped onto a bounded $\mathcal{C}$-set in Euclidean space by a semialgebraic bijection.
\end{itemize}
Furthermore, for $\mathcal{C=C(O)}$ we get back the o-minimal structure $\mathcal{O}$ by this correspondence, and for $\mathcal{O=O(C)}$ we again get $\mathcal{C}$.
\end{theorem}
\begin{proof}
See \cite{Dries1996} and \cite[Theorem 1.1.3]{Lageman2007}.
\end{proof}
As a consequence of the correspondence between o-minimal structures containing $\mathbb{R}_{an}$ and analytic-geometric categories, the definable sets associated allow us to provide examples of $\mathcal{C}$-sets in $\mathcal{C(O)}$. Furthermore, C-functions are locally mapped to definable functions by analytic charts.
\begin{proposition}\label{prop2} 
Let $f:M\to\mathbb{R}$ be a $\mathcal{C}$-function and $\phi: U\rightarrow \mathbb{R}^n$, $U\subset M$ an analytic local chart. Assume that $U\subset\mbox{dom}f$ and $V\subset M$ is a bounded open set such that $\overline{V}\subset U$. If $f$ restricted to $U$ is a bounded $\mathcal{C}$-function, then 
\[
f\circ \phi^{-1}: \phi(V)\to\mathbb{R},
\]
is definable in $\mathcal{O(C)}$.
\end{proposition}
\begin{proof}
See \cite[Proposition 1.1.5.]{Lageman2007}.
\end{proof}

The next result provided us a nonsmooth extension of the Kurdyka-Lojasiewicz inequality for $\mathcal{C}$-functions defined on analytic manifolds.
\begin{theorem}
Let $M$ be a analytic Riemannian manifold and $f:M\to \mathbb{R}$ a continuous $\mathcal{C}$-function. Then, $f$ is a KL function.  Moreover, the function $\varphi$ which appears in \eqref{eq:kur100} is definable in $\mathcal{O}$.
\end{theorem}
\begin{proof}
Take $\bar{x}\in M$ a critical point of $f$ and let $\phi:V\to\mathbb{R}^n$ be an analytic local chart with $V\subset M$ a neighbourhood of $\bar{x}$ chosen such that $V$ and $f(V)$ are bounded. Thus, from Proposition~\ref{prop2}, we have that $h=f\circ \phi^{-1}:\phi(V)\to\mathbb{R}$ is a definable function in $\mathcal{O(C)}$. Thus, as $\phi(V)$ is a bounded open definable set containing $\bar{y}=\phi(\bar{x})$ and $\phi$ is definable, applying Theorem 11 of \cite{Bolte2007} with $U=\phi(V)$ and taking into account Theorem 4.1 of \cite{Attouch2010-2}, Kurdyka-Lojasiewicz inequality holds at $\bar{y}=\phi(\bar{x})$, i.e., there exists $\eta \in (0,+\infty]$ and a continuous concave function $\Phi:[0,\eta)\to \mathbb{R}_+$ such that:
\begin{itemize}
\item [(i)] $\Phi(0)=0$, $\Phi\in C^1(0,\eta)$ and, for all $s\in(0,\eta)$, $\Phi '(s)>0;$
\item [(ii)] for all $y\in U\cap [h(\bar{y})<h<h(\bar{y})+\eta]$, it holds
\[
\Phi '(h(y)-h(\bar{y}))dist(0, \partial h(y))\geq 1.
\]
\end{itemize}
Since $\phi$ is a diffeomorphism and using that $y=\phi(x)$, $\bar{y}=\phi(\bar{x})$ and $h=f\circ\phi^{-1}$, from Proposition~\ref{regracadeiaZhu} last inequality yields 
\[
\Phi '(f(x)-f(\bar{x}))dist(0, (\phi^*_x)^{-1}\partial f(x))\geq 1,\quad x\in V\cap [0<f<f(\bar{x})+\eta],
\]
where $\phi^*_x$ denote the Fréchet derivative adjunct of the function $\phi$.

Take $V'\subset V$ an open set such that $K=\overline{V'}$ is contained in the interior of the set V and $\bar{x}\in V'$.
Thus, $K$ is a compact set and for each $x\in K$ there exists $c_x>0$ with
\[
\|(\phi^*_x)^{-1}w\|\leq c_x\|w\|,\qquad w\in T_xM.
\]
Since $K$ is a compact set and $(\phi^*_x)^{-1}$ is a diffeomorphism, there exists a positive constant $c:=\sup\{c_x: x\in K\}$ such that
\[
\|(\phi^*_x)^{-1}w\|\leq c\|w\|,\qquad w\in T_xM,\;  x\in K.
\]
Hence, for $x\in V' \cap [0<f<f(\bar{x})+\eta]$, we have
\[
1\leq \Phi'(f(x)-f(\bar{x}))dist(0, (\phi^*_x)^{-1}\partial f(x))\leq c\,\Phi '(f(x)-f(\bar{x}))dist(0, \partial f(x)),
\]
and the Kurdyka-Lojasiewicz inequality holds at $\bar{x}$ with $\varphi=c\,\Phi$. Therefore, combining arbitrary of $\bar{x}$ with Lemma~\ref{Lem3.2} we conclude that $f$ is a KL function. The second part also follows from Theorem 11 of \cite{Bolte2007} and the proof is concluded.

\end{proof}

The following result provided us a nonsmooth extension of the Kurdyka-Lojasiewicz inequality for definable functions defined on submanifolds of Euclidean space. Coste \cite{Coste2000} devotes Chapter $6$ to establish properties of such submanifolds.

\begin{theorem}
Let $f:M\subset\mathbb{R}^n\to\mathbb{R}\cup\{+\infty\}$ be proper lower semicontinuous definable function in an o-minimal structure $\mathcal{O}$. If $M$ is endowed with the induced metric of Euclidean space, then $f$ is a KL function. Moreover, the function $\varphi$ which appears in \eqref{eq:kur100} is definable in $\mathcal{O}$.
\end{theorem}
\begin{proof}
Take $\bar{x}\in M$ a critical point of $f$ and $W$ a bounded definable subset of $\mathbb{R}^n$ such that $\bar{x}\in W$. Since $\mbox{dom}f$ and $W$ are definable sets in $\mathbb{R}^n$ and $W$ is bounded, it follows that $\mbox{dom}f\cap W$ is a bounded definable set in $\mathbb{R}^n$. Thus, applying Theorem 11 of \cite{Bolte2007} with $U=\mbox{dom}f\cap W$ and, taking into account Theorem 4.1 of \cite{Attouch2010-2}, the Kurdyka-Lojasiewicz inequality holds at $\bar{x}$. Therefore, combining arbitrary of $\bar{x}$ with Lemma~\ref{Lem3.2}, we conclude the first part of the theorem. The second part also follows, of Theorem 11 of \cite{Bolte2007} and the proof is concluded.
\end{proof}
\begin{remark}
A large class of examples of definable submanifolds of Euclidean space are given by manifolds which are obtained as reverse image of regular value of a definable function, more precisely, if $F:\mathbb{R}^{n+k}\to \mathbb{R}^{k}$ is a $C^p$ definable function and $"0"$ is a regular value of $F$, then $M=F^{-1}(0)$ is a definable submanifold of $\mathbb{R}^n$. Moreover, via the Nash Theorem (\cite{Nash1956}), we can isometrically imbed in some $\mathbb{R}^n$ a small piece $\mathcal{Y}$ of $M$, which is a regular submanifold of $\mathbb{R}^n$. Indeed, if $\epsilon>0$ is small enough, then the set of normal segments of radius $\epsilon$ centered at points of $\mathcal{Y}$ determine a tubular neighbourhood $\mathcal{V}$ of $\mathcal{Y}$. Clearly, $\mathcal{V}$ has a natural coordinate system given by $y=(x, t)\in\mathcal{Y}\times B_{\epsilon}(0)$, 
where $B_{\epsilon}(0)\subset\mathbb{R}^m$ is an $\epsilon$-ball (here, $n-m,\; m < n$, is the dimension of $M$). We identify $(x, 0)$ with $x$. Define $h(x,t)=t$. It is obvious that $h$ is a definable function and $\mathcal{Y}=\{y\in V; h(y)=0\}$ is a definable submanifold of $\mathbb{R}^n$.
\end{remark}

\section {An abstract convergence result for inexact descent methods}\label{sec5}

In this section we present an abstract converge analysis of inexact descent methods for functions satisfying Kurdyka-Lojasiewicz inequality at a given critical point. Throughout this section $f$ denotes a proper lower semicontinuous function $f:M\to\mathbb{R}\cup\{+\infty\}$. 

Next we present the definition of quasi distance.
\begin{definition}\label{qd}
A mapping $D: M\times M\to \mathbb{R}_+$ is said to be a quasi distance if:
\item [(i)] for all $x, y\in M, \ D(x,y)= D(y,x)=0 \Leftrightarrow x=y$;
\item[(ii)] for all $x, y, z \in M, \ D(x,z)\leq D(x,y)+D(y,z)$.
\end{definition}

If $D$ is also symmetric, that is, $D(x, y) = D(y, x)$, $x, y \in M$, then $D$ is
a distance. Given $x\in M$ and $\epsilon>0$ fixed, we denote by $B_D(x,\epsilon)$ the ball with respect to the quasi distance $D$ defined by
\[
B_D(x,\epsilon)=\{y\in M: D(y,x)<\epsilon\}.
\]

Throughout this section we assume that, for each $y\in M$, $D(.,y)$ is continuous.

Let $a$ and $b$ be fixed positive constants and $\{x^k\}$ an arbitrary sequence satisfying the following assumptions,
\begin{itemize}
\item[{\bf H1.}] For each $k\in\mathbb{N}$, 
\[
f(x^{k+1})+a D^2(x^{k+1}, x^k)\leq f(x^k);
\]
\item[{\bf H2.}] For each $k\in\mathbb{N}$, there exists $w^{k+1}\in\partial f(x^{k+1})$ such that
\[
\|w^{k+1}\|\leq bD(x^{k+1},x^k);
\]
\item[{\bf H3.}] $f$ restricted to \mbox{dom}$f$ is continuous;

\item[\bf H4.] $\sum_{k=0}^{+\infty}D(x^{k+1},x^k)<+\infty$ implies that $\{x^k\}$ is convergent on $M$. 

\end{itemize}
 \begin{remark}\label{remark1000}
From assumption {\bf H2} it is immediate that if $x^{k+1}=x^k$ for some $k$, then $x^k$ is a critical point of $f$. Note also that if $D$ is Riemannian distance then assumption {\bf H4} holds (Hopf-Rinow's theorem). However, even in the Euclidean case there exists a particular class of quasi distances that not necessary are distances but satisfy the assumption {\bf H4}, see \cite{Moreno2011}.
 \end{remark}
 
From now on, in this section, we assume that $\{x^k\}$ is a sequence satisfying assumptions {\bf H1}, {\bf H2}, {\bf H3} and {\bf H4}. Moreover, taking into account the first part of Remark~\ref{remark1000}  we assume that $x^{k+1}\neq x^k$ for all $k$.

Next, we present one technical result that could be useful in convergence analysis.
\begin{lemma}\label{lem3}
Let $\{a_k\}$ be a sequence of positive numbers such that
\[
\sum_{k=1} ^{+\infty} {a_k^2}/{a_{k-1}} <+\infty.
\]
Then, $\sum_{k=1} ^{+\infty} a_k < + \infty$.  
\end{lemma}
\begin{proof}
Take $j\in \mathbb{N}$ fixed. Note that,
\[
\sum_{k=1} ^j a_k =\sum_{k=1} ^j \frac{a_k}{ \sqrt{a_{k-1}}}\sqrt{a_{k-1}}\leq\left(\sum_{k=1} ^j \frac{a_k ^2}{{a_{k-1}}}\right)^{1/2}\left(\sum_{k=1} ^j {a_{k-1}}\right)^{1/2}, 
\]
where the above inequality follows from Cauchy-Schwartz inequality in $\mathbb{R}^j$ with respect to the vectors $\left(a_1/ \sqrt{a_0},\ldots,  a_j/\sqrt{a_{j-1}}\right)$ and $\left(\sqrt{a_{0}},\ldots,\sqrt{a_{j-1}}\right)$. So, 
\[
\sum_{k=1} ^j a_k\leq\left(\sum_{k=1} ^j \frac{a_k ^2}{{a_{k-1}}}\right)^{1/2}\left(\sum_{k=1}^j {a_{k-1}}\right)^{1/2}.
\]
Now, adding $a_0$ to both sides of the last inequality and taking into account that $a_j>0$, we obtain 
\[
\sum_{k=1} ^j a_{k-1}\leq a_0+\left(\sum_{k=1} ^j \frac{a_k ^2}{{a_{k-1}}}\right)^{1/2}\left(\sum_{k=1}^j {a_{k-1}}\right)^{1/2}.
\]
Therefore, dividing both sides of last inequality by $\left(\sum_{k=1}^j {a_{k-1}}\right)^{1/2}$ and observing that 
\[
a_0/\left(\sum_{k=1}^j {a_{k-1}}\right)^{1/2}\leq\sqrt{a_0}\qquad (a_k>0, k=0,1,\ldots),
\]
it follows that
\[
\left(\sum_{k=1}^j {a_{k-1}}\right)^{1/2}\leq \sqrt{a_0}+\left(\sum_{k=1} ^j \frac{a_k ^2}{{a_{k-1}}}\right)^{1/2},
\]
and the desired result follows by using simple arguments of real analysis.
\end{proof}

In the following theorem we prove the full convergence of the sequence $\{x^k\}$ to a critical point of functions which satisfy the Kurdyka-Lojasiewicz property in that point.

\begin{theorem}\label{theoconverg}
Let $U$, $\eta$ and $\varphi:[0,\eta)\to\mathbb{R}_+$ be the objects appearing in the Definition \ref{dkl}. Assume that $x^0\in\mbox{dom}f$, $\tilde{x}\in M$ is an accumulation point of the sequence $\{x^k\}$, $\rho>0$ is such that $B_D(\tilde{x},\rho)\subset U$ and $f$ satisfies the Kurdyka-Lojasiewicz inequality at $\tilde{x}$.
Then there exists $k_0\in\mathbb{N}$ such that 
\begin{equation}\label{eq:somaconver}
\sum_{k=k_0}^{+\infty}D(x^{k+1},x^{k})<+\infty.
\end{equation}
Moreover, $f(x^k)\to f(\tilde{x})$, as $k\to+\infty$,
and the sequence $\{x^k\}$ converges to $\tilde{x}$ which is a critical point of $f$.
\end{theorem}

In order to prove the above theorem we need some preliminary results over which we assume that all assumptions of Theorem 4.1 hold, with the exception of {\bf H1}, {\bf H2}, {\bf H3} and {\bf H4}, which will be assumed to hold only when explicitly stated.

\begin{lemma}\label{lemma:conver1}
Assume that assumptions {\bf H1} and {\bf H3} hold. Then there exists $k_0\in\mathbb{N}$ such that
\begin{equation}\label{eq:conv1}
f(\tilde{x})< f(x^{k})<f(\tilde{x})+\eta,\qquad k\geq k_0,
\end{equation}
\begin{equation}\label{eq:conv2}
D(x^{k_0},\tilde{x})+2\sqrt{\frac{f(x^{k_0})-f(\tilde{x})}{a}}+\frac{b}{a}\varphi(f(x^{k_0})-f(\tilde{x}))<\rho.
\end{equation}
Moreover, if {\bf H2} holds, then
\begin{equation}\label{eq:conv6}
\frac{b}{a}[\varphi(f(x^{k_0})-f(\tilde{x}))-\varphi(f(x^{k_0+1})-f(\tilde{x}))]\geq \frac{D^2(x^{k_0+1},x^{k_0})}{D(x^{k_0},x^{k_0-1})}.
\end{equation}
In particular, if $x^k\in B_D(\tilde{x},\rho)$ for all $k\geq k_0$, then  $\sum_{k=k_0}^{+\infty}D(x^{k+1},x^{k})<+\infty$ and, assuming that {\bf H4} holds, the sequence $\{x^k\}$ converges to $\tilde{x}$.
\end{lemma}

\begin{proof}
Let $\{x^{k_j}\}$ be a subsequence of $\{x^k\}$ converging to $\tilde{x}$. Now, from assumption {\bf H1} combined with $x^0\in\mbox{dom}f$ and $a>0$, we obtain that $x^k_j\in\mbox{dom}f$, for all $j\in\mathbb{N}$ (in particular $\tilde{x}\in\mbox{dom}f$). Thus, from assumption {\bf H3} and taking into account that $\lim_{s\to+\infty}x^{k_j}=\tilde{x}$,  it follows that $\{f(x^{k_j})\}$ converge to $f(\tilde{x})$. Since $\{f(x^k)\}$ is a decreasing sequence (it holds trivially of the assumption {\bf H1}), we obtain that the whole sequence $\{f(x^k)\}$ converges to $f(\tilde{x})$ as $k$ goes to $+\infty$ and, hence 
\begin{equation}\label{eq:conv4}
f(\tilde{x})< f(x^{k}),\qquad k\in\mathbb{N}.
\end{equation}
In particular, there exists $N\in \mathbb{N}$ such that 
\begin{equation}\label{eq:tec3}
f(\tilde{x})< f(x^{k})<f(\tilde{x})+\eta,\qquad k\geq N.
\end{equation}
Since \eqref{eq:conv4} holds, let us define the sequence $\{b_k\}$ given by
\[
b_k=D(x^k,\tilde{x})+2\sqrt{\frac{f(x^k)-f(\tilde{x})}{a}}+\frac{b}{a}\varphi(f(x^k)-f(\tilde{x})).
\]
As $D(.,\tilde{x})$ and $\varphi$ are continuous it follows that $0$ is an accumulation point of the sequence $\{b_k\}$ and hence there exists $k_0:=k_{j_0}>N$ such that \eqref{eq:conv2} holds. In particular, as $k_0>N$, from \eqref{eq:tec3} it also holds \eqref{eq:conv1}.

From \eqref{eq:conv1} combined with $x^{k_0}\in B_D(\tilde{x},\rho)$ (it follows from \eqref{eq:conv2}), we have that
\[
x^{k_0}\in B_D(\tilde{x},\rho)\cap [f(\tilde{x})<f<f(\tilde{x})+\eta].
\]
So, since $\tilde{x}$ is a point where $f$ satisfies the Kurdyka-Lojasiewicz inequality it follows that $0\notin\partial f(x^{k_0})$. Moreover, assumption {\bf H2} combined with the definition of $dist(0,\partial f(x^k))$, yields 
\[
bD(x^k,x^{k-1})\geq \|w^k\|\geq dist(0,\partial f(x^k)),\qquad k=1,2,\ldots.
\]
Thus, again from the Kurdyka Lojasiewicz inequality of $f$ at $\tilde{x}$, it follows that
\begin{equation}\label{eq:conv7}
\varphi'(f(x^{k_0})-f(\tilde{x}))\geq\frac{1}{bD(x^{k_0},x^{k_0-1})}.
\end{equation}
On the other hand, the concavity of the function $\varphi$ implies that
\[
\varphi(f(x^{k_0})-f(\tilde{x}))-\varphi(f(x^{k_0+1})-f(\tilde{x}))\geq \varphi'(f(x^{k_0})-f(\tilde{x}))(f(x^{k_0})-f(x^{k_0+1})),
\]
which, combined with $\varphi'>0$ and assumption {\bf H1} yields
\[
\varphi(f(x^{k_0})-f(\tilde{x}))-\varphi(f(x^{k_0+1})-f(\tilde{x}))\geq \varphi'(f(x^{k_0})-f(\tilde{x}))aD^2(x^{k_0+1},x^{k_0}).
\]
Therefore, \eqref{eq:conv6} follows by combining the last inequality with \eqref{eq:conv7}.  

The proof of the latter part follows from \eqref{eq:conv6} combined with Lemma~\ref{lem3} and assumption {\bf H4}, which concludes the proof of the lemma.
\end{proof}

\begin{lemma}\label{lemmatheo2}
Assume that assumptions {\bf H1}, {\bf H2} and {\bf H3} hold. Then, there exists a $k_0\in \mathbb{N}$ such that
\begin{equation}\label{eq:conv10}
x^k\in B_D(\tilde{x},\rho), \qquad k> k_0.
\end{equation}
\end{lemma}
\begin{proof}
The proof is by induction on $k$. It follows trivially from the {\bf H1} that sequence $\{f(x^k)\}$ is decreasing and
\begin{equation}\label{eq:conv9}
D(x^{k+1},x^k)\leq \sqrt{\frac{f(x^{k})-f(x^{k+1})}{a}}, \qquad k\in \mathbb{N}.
\end{equation}  
Moreover, as {\bf H3} also holds, from Lemma~\ref{lemma:conver1} it follows that there exists $k_0\in\mathbb{N}$ such that \eqref{eq:conv2}, \eqref{eq:conv1} hold and, hence
\begin{equation}\label{eq:conv21}
x^{k_0}\in B_D(\tilde{x},\rho),\qquad 0<f(x^{k_0})-f(x^{k_0+1})< f(x^{k_0})-f(\tilde{x}),
\end{equation}
which, combined with \eqref{eq:conv9} $(k=k_0)$, give us
\begin{equation}\label{eq:conv14}
D(x^{k_0+1},x^{k_0})\leq \sqrt{\frac{f(x^{k_0})-f(\tilde{x})}{a}}.
\end{equation}
Now, from the triangle inequality, combining with the last expression and \eqref{eq:conv2}, we obtain
\[
D(x^{k_0+1},\tilde{x})\leq \sqrt{\frac{f(x^{k_0})-f(\tilde{x})}{a}} + D(x^{k_0},\tilde{x})< \rho,
\]
which implies that $x^{k_0+1}\in B_D(\tilde{x},\rho)$. 

Suppose now that \eqref{eq:conv10} holds for all $k=k_0+1,\ldots,k_0+j-1$. In this case, for $k=k_0+1,\ldots,k_0+j-1$, it holds \eqref{eq:conv6} and, consequently
\begin{equation}\label{eq:conv12}
\sqrt{D(x^{k},x^{k-1})(b/a)[\varphi(f(x^k)-f(\tilde{x}))-\varphi(f(x^{k+1})-f(\tilde{x}))]}\geq D(x^{k+1},x^k).
\end{equation}
Thus, since for $r,s\geq 0$ it holds $r+s\geq 2\sqrt{rs}$, considering, for $k=k_0+1,\ldots,k_0+j-1$ 
\[
r=D(x^k, x^{k-1}),\;s=(b/a)[\varphi(f(x^k)-f(\tilde{x}))-\varphi(f(x^{k+1})-f(\tilde{x}))],
\]
from the inequality \eqref{eq:conv12}, it follows, for $k=k_0+1,\ldots,k_0+j-1$, that
\[
2D(x^{k+1},x^k)\leq  D(x^k, x^{k-1})+\frac{b}{a}[\varphi(f(x^k)-f(\tilde{x}))-\varphi(f(x^{k+1})-f(\tilde{x}))].
\]
So, adding member to member, with $k=k_0+1,\ldots,k_0+j-1$, we obtain
\begin{multline*}
\sum_{i=k_0+1}^{k_0+j-1}D(x^{i+1},x^i)+D(x^{k_0+j},x^{k_0+j-1})\leq  D(x^{k_0+1}, x^{k_0})+\frac{b}{a}[\varphi(f(x^{k_0+1})-f(\tilde{x}))\\ -\varphi(f(x^{k_0+j})-f(\tilde{x}))],
\end{multline*}
from which we obtain
\begin{equation}\label{eq:conv13}
\sum_{i=k_0+1}^{k_0+j-1}D(x^{i+1},x^i)\leq  D(x^{k_0+1}, x^{k_0})+\frac{b}{a}\varphi(f(x^{k_0+1})-f(\tilde{x}))\leq D(x^{k_0+1}, x^{k_0})+\frac{b}{a}\varphi(f(x^{k_0})-f(\tilde{x}))
\end{equation}
where the last inequality follows of the second inequality in \eqref{eq:conv21} and because $\varphi$ is increasing.
Now, using the triangle inequality and taking into account that $D(x,y)\geq 0$ for all $x,y\in M$, we have
\[
D(x^{k_0+j},\tilde{x})\leq D(x^{k_0+j},x^{k_0})+D(x^{k_0},\tilde{x})\leq D(x^{k_0},\tilde{x})+D(x^{k_0+1},x^{k_0})+\sum_{i=k_0+1}^{k_0+j-1}D(x^{i+1},x^i),
\]
which, combined with \eqref{eq:conv13}, yields
\[
D(x^{k_0+j},\tilde{x})\leq D(x^{k_0},\tilde{x})+2D(x^{k_0+1},x^{k_0})+\frac{b}{a}\varphi(f(x^{k_0})-f(\tilde{x})).
\]
Therefore, from the last inequality, combined with \eqref{eq:conv14} and \eqref{eq:conv2}, we conclude that $x^{k_0+j}\in B_D(\tilde{x},\rho)$, which completes the induction proof.
\end{proof}

\subsubsection*{Proof of Theorem~\ref{theoconverg}}
Note that Lemma~\ref{lemmatheo2} combined with Lemma~\ref{lemma:conver1} implies that \eqref{eq:somaconver} holds and, in particular, that the sequence $\{x^k\}$ converge to $\tilde{x}\in M$. Thus, from the assumption {\bf H3} combined with assumption {\bf H1}, it follows $f(x^k)\to f(\tilde{x})$, as $k\to+\infty$. Now, combining \eqref{eq:somaconver} with assumption {\bf H2}, it follows that $\{(x^k,w^k)\}$ converges $(\tilde{x},0)$ as $k$ goes to $+\infty$. Therefore, from Definition~\ref{eq:sublimiting} we conclude that $0\in \partial f(\tilde{x})$, which tell us that $\tilde{x}$ is a critical point of $f$.

\section{Inexact proximal method for KL functions on Riemannian manifold}\label{sec6}
In this section we recall the exact proximal point method in the Riemannian context proposed by Ferreira and Oliveira $\cite{FO2000}$ and propose an inexact version of it in that context. 

Consider the following optimization problem
\begin{eqnarray}\label{po:conv}
\begin{array}{clc}
  & \min  f(x) \\
   & \textnormal{s.t.}\,\,\, x\in M,\\
\end{array}
\end{eqnarray}
where $f:M\to\mathbb{R}\cup\{+\infty\}$ is a proper lower semicontinuous function bounded from below. 

The proximal point method to solve optimization problems of the
form \eqref{po:conv} generates, for a starting point $x^0\in M$, a sequence $\{x^k\}\subset M$ as it follows:
\begin{equation}\label{eq:prox1_1}
x^{k+1}\in \mbox{argmin}_{y\in M}\left\{f(y)+\frac{\lambda_k}{2} d^2(y,x^k)\right\},
\end{equation}
where $\{\lambda_k\}$ is a sequence of positive numbers.  In the particular case that $M$ is a Hadamard manifold, $\mbox{dom}f=M$ and $f$ is convex, Ferreira and Oliveira \cite{FO2000} proved that for each $k\in \mathbb{N}$ the function $f(.)+d^2(.,x^k):M\to\mathbb{R}$ is 1-coerciva and, consequently, the well definedness of the sequence $\{x^k\}$ with $x^{k+1}$ being uniquely determined. Moreover, considering $\sum_{k=0}^{+\infty}1/\lambda_k=+\infty$ and that $f$ has minimizer, the authors proved convergence of the sequence $\{f(x^k)\}$ to the minimum value and convergence from the sequence $\{x^k\}$ to a minimizer point. Note that from \eqref{eq:prox1_1} combined with the assumption of convexity of the function $f$ and first order optimality condition associated to the subproblem \eqref{eq:prox1_1}, 
\begin{equation}\label{eq:prox0001}
f(x^{k+1})+\frac{\lambda_k}{2}d^2(x^{k+1},x^k)\leq f(x^k),
\end{equation}
\begin{equation}\label{eq:prox2}
\forall\; k\in\mathbb{N},\quad \exists\; w^{k+1}\in\partial f(x^{k+1}),
\end{equation}
\begin{equation}\label{eq:prox3}
\|w^{k+1}\|= \lambda_kd(x^{k+1},x^k).
\end{equation} 

Next we present an inexact version of the proximal point method in the Riemannian context.

Take $x^0\in \mbox{dom}f$, $0<\bar{\lambda}\leq\tilde{\lambda}<+\infty$, $b>0$ and $\theta\in (0,1]$. For each $k=0,1,\ldots$, choose $\lambda_k\in [\bar{\lambda},\tilde{\lambda}]$, and find $x^{k+1}\in M$, $w^{k+1}\in T_{x^{k+1}}M$ such that
\begin{equation}\label{eq:prox4}
f(x^{k+1})+\frac{\theta\lambda_k}{2}D^2(x^{k+1},x^k)\leq f(x^k),
\end{equation}
\begin{equation}\label{eq:prox5}
\forall\; k\in\mathbb{N},\quad \exists\; w^{k+1}\in\partial f(x^{k+1}),
\end{equation}
\begin{equation}\label{eq:prox6}
\|w^{k+1}\|\leq b\lambda_kD(x^{k+1},x^k),
\end{equation} 
where $D:M\times M\to\mathbb{R}$ is a quasi distance continuous.
\begin{remark}
Note that if $\theta=b=1$, \eqref{eq:prox6} holds with equal, $D$ is the Riemannian distance $d$, $M$ is a Hadamard manifold, $\mbox{dom}f=M$ and $f$ is convex, and we recover the exact proximal point method generated by \eqref{eq:prox0001}, \eqref{eq:prox2} and \eqref{eq:prox3}. On the other hand, if $M=\mathbb{R}^n$ and $D$ is the Euclidean distance, we are with the inexact version of the proximal point method proposed by Attouch et al., in \cite{Attouch2010}.
\end{remark}

\begin{theorem}
Let $f:M\to\mathbb{R}\cup\{+\infty\}$ be a proper lower semicontinuous KL function which is bounded from below. Moreover, assume that assumption {\bf H3} holds. If a sequence $\{x^k\}$ generated by \eqref{eq:prox4}, \eqref{eq:prox5} and \eqref{eq:prox6} is bounded and $D$ satisfies assumption {\bf H4}, then it converges to some critical point $\bar{x}$ of $f$.
\end{theorem}
\begin{proof}
Since $\{x^k\}$ is bounded, by Roph-Rinow's theorem the sequence $\{x^k\}$ has an accumulation point on M. Let $\bar{x}$ be an accumulation point of $\{x^k\}$ and $\{x^{k_j}\}$ a subsequence converging to $\bar{x}$. Now, from \eqref{eq:prox4} combined with $x^0\in\mbox{dom}f$ and $\theta\in (0,1]$, we obtain that $x^k_j\in\mbox{dom}f$, for all $j\in\mathbb{N}$ (in particular $\bar{x}\in\mbox{dom}f$). Therefore, as assumption {\bf H1} holds with $a=(\theta\bar{\lambda})/2$, assumption {\bf H2} holds by the definition of the method and assumptions {\bf H3} and {\bf H4} hold by the assumption of the theorem, the result follows by directly applying Theorem~\ref{theoconverg}.   
\end{proof}

\section{Inexact descent method for KL functions on Riemannian manifold}\label{sec7}

In this section we recall the gradient method in the Riemannian context to solve optimization problems of the form \eqref{po:conv} in the case that $\mbox{dom}f=M$ and $f$ is a $C^1$ function, and extend the convergence results established by da Cruz Neto et al. \cite{XLO1998} and Papa Quiroz et al. \cite{PP08}, respectively, for convex and quasiconvex functions on Riemannian manifolds of positive curvature, for KL functions on Hadamard manifolds. 

\subsection{The steepest descent method with some known stepsize rule}
Given $x^0\in M$, the classic steepest descent method generates a sequence $\{x^k\}$ given by
\begin{equation}\label{eq:sdm1_1}
x^{k+1}=\exp_{x^k}(-t_k\mbox{grad} f(x^k)),
\end{equation}
where exp is the exponential map and $t_k$ is some positive stepsize. 

\noindent{\bf Armijo search}:
 If the sequence $\{t_k\}$ is obtained by
\begin{equation} \label{eq:sl}
 t_k := \max \left\{2^{-j} : j\in {\mathbb N},\, f \left(exp _{x^{k}}(2^{-j}\mbox{grad} f(x^k)\right)\leq f(x^{k})- \alpha 2^{-j}\,\|\mbox{grad} f(x^{k})\|^2\right\},
\end{equation}
with $\alpha\in (0,1)$, we are with the Armijo search. Note that, in this case, zero can be an accumulation point of the sequence $\{t_k\}$. However, when $f$ is $L$-Lipschitzian gradient (see next definition) zero is not accumulation point of the sequence $t_k$.

The next definition was proposed by da Cruz Neto et al. \cite{XLO1998}.

\begin{definition}
Let $f:M\to\mathbb{R}$ be a function $C^1$ and $L>0$. $f$ is said to have $L$-Lipschitzian gradient if, for any $x,y\in M$ and any geodesic segment $\gamma:[0,r]\to M$ joining $x$ and $y$, we have
\[
\|\mbox{grad}f(\gamma(t))-P_{\gamma(a)\gamma(t)}\mbox{grad}f(x)\|\leq r l(t), \qquad t\in [0,r],
\]
where $l(t)$ denotes the length of the segment between $\gamma (0)=x$ and $\gamma(t)$. In particular, if $M$ is a Hadamard manifold, then the last inequality reduces to
\[
\|\mbox{grad}f(\gamma(t))-P_{x\gamma(t)}\mbox{grad}f(x)\|\leq r d(\gamma(t),x), \qquad t\in [0,r].
\]
\end{definition}
\noindent{\bf Fixed step}(See Burachik et al. \cite{Bur1995} and da Cruz Neto et al. \cite{XLO1998})

Given $\delta_1,\delta_2>0$ such that $L\delta_1+\delta_2<1$, where $L$ is the Lipschtz constant associated to $\mbox{grad}f$, if sequence $\{t_k\}$ is such that
\[
t_k\in\left(\delta_1,\frac{2}{L}(1-\delta_2)\right),
\]
we are with the fixed step rule.

Let us now consider the following assumption:

\begin{assumption}\label{eq:sdm4}
There exists a function $\phi:\mathbb{R}_+\to\mathbb{R}_+$ such that
\begin{itemize}
\item [a)] There exist $\alpha\in (0,1)$ and $\tau_{\alpha}>0$, such that $\forall t\in (0,\tau_{\alpha}]$, $\phi(t)\leq \alpha t$;
\item [b)] There exist $\beta>0$ and $\tau_{\beta}\in (0,+\infty]$, such that $\forall t\in (0,\tau_{\beta}]\cap\mathbb{R}$, $\phi(t)\geq \beta t^2$;
\item [c)] For all $k=0,1,\ldots$, $f(x^{k+1})\leq f(x^k)-\phi(t_k)\|\mbox{grad}f(x^k)\|^2$ and $0<t_k\leq \tau_{\beta}$ in \eqref{eq:sdm1_1};
\item [d)] There exist $\gamma>1$, $\tau_{\gamma}>0$, such that $\forall k$, $t_k\geq \tau_{\gamma}$ or  
\[
\mbox{there exists}\quad \bar{t}_k\in [t_k,\gamma t_k ]:\; f(\exp_{x^k}(-\bar{t}_k\mbox{grad}f(x^{k}))\geq f(x^k)-\phi(\bar{t}_k)\|\mbox{grad}f(x^k)\|^2.
\]
\end{itemize}
\end{assumption}
\begin{remark}
The above assumption was first considered by Kiwiel et al. \cite{Kiwiel1996} in the Euclidean context. They observed that the steepest descent methods with Armijo search and fixed step both satisfy Assumption~\ref{eq:sdm4} with
\[
\phi(t)=\alpha t,\quad \beta=\alpha,\quad \gamma=2\quad \mbox{and}\quad \tau_{\alpha}=\tau_{\beta}=\tau_{\gamma}=1
\] 
and
\[
\phi(t)=\beta t^2,\quad \beta=\frac{\delta_2 L}{2(1-\delta_2)},\quad \tau_{\gamma}=\delta_1,\quad \tau_{\beta}=\frac{L}{2(1-\delta_2)},\quad \alpha\in (0,1)\quad \tau_{\alpha}=\alpha/\beta,
\] 
respectively. Under Assumption~\ref{eq:sdm4}, quasi convexity of the function $f$ and that the solution set of the problem \eqref{po:conv} is not-empty, Kiwiel et al., proved full convergence of the sequence generated by the method to a critical point. Papa Quiroz et al. \cite{PP08}, considering Assumption~\ref{eq:sdm4}, generalized the convergence result presented in \cite{Kiwiel1996} to the Riemannian context in the particular case that M has nonnegative curvature. 
\end{remark}

Considering a sequence $\{x^k\}$, generated by \eqref{eq:sdm1_1}, satisfying Assumption~\ref{eq:sdm4} with the following reformulation of the item $d)$,
\begin{equation}\label{eq:assumpmod1}
\exists\; \gamma>1,\; \tau_{\gamma}>0:\; t_k\geq \tau_{\gamma},\qquad k=0,1,\ldots,
\end{equation}
then, we have at least the steepest descent method with Armijo search and with fixed step (in the case that the objective function is $L$-Lipschitzian gradient) which satisfies those assumptions.
\subsection{General convergence result}
Next we present a general descent method to solve the optimization problem \eqref{po:conv}. From now on, in this section $f$ denotes a $C^1$ function with $L$-Lipschitz gradient.

Given $r_1, r_2>0$, and $x^0\in M$, consider the sequence $\{x^k\}$ generated as follows: 
\begin{equation}\label{eq:sdm2}
f(x^{k+1})+r_1 D^2(x^{k+1}, x^k)\leq f(x^k);
\end{equation}
\begin{equation}\label{eq:sdm3}
\|\mbox{grad}f(x^k)\|\leq r_2D(x^{k+1},x^k),
\end{equation}
where $D:M\times M\to\mathbb{R}$ is a quasi distance.
Note that if $\{x^k\}$ is generated by \eqref{eq:sdm1_1}, then \eqref{eq:sdm3} does not necessarily happens. This is due to the fact that the geodesic through $x^k$ with velocity $-\mbox{grad}f(x^k)$ is not necessarily minimal.  
 
The following lemma provides us a class of sequences which falls into the general category delineated by the general descent method \eqref{eq:sdm2} and \eqref{eq:sdm3}.

\begin{lemma}
Let $\{x^k\}$ be the sequence generated by \eqref{eq:sdm1_1} satisfying Assumption~\ref{eq:sdm4} with the item $d)$ replaced by condition \eqref{eq:assumpmod1}. Assume that there exists $s_1,s_2>0$ such that
\begin{equation}\label{equ:quasidist10}
s_1d(x,y)\leq D(x,y)\leq s_2 d(x,y),\qquad x, y\in M.
\end{equation}
Then $\{x^k\}$ satisfies \eqref{eq:sdm2}. Moreover, if $M$ is a Hadamard manifold, then $\{x^k\}$ also satisfies \eqref{eq:sdm3}.
\end{lemma}
\begin{proof}
From items b) and c) of Assumption~\ref{eq:sdm4}, we obtain
\[
f(x^{k+1})+\beta t_k\|\grad f(x^k)\|^2\leq f(x^k).
\]
Now, from \eqref{eq:sdm1_1} it follows that $d(x^{k+1},x^k)\leq t_k\|\grad f(x^k)\|$ which, combined with last inequality and \eqref{equ:quasidist10}, yields
\[
f(x^{k+1})+\frac{\beta}{s_2}D^2(x^{k+1},x^k)\leq f(x^k).
\]
Thus, \eqref{eq:sdm2} it holds with $r_1=\beta/ s_2$. In the particular case that $M$ is a Hadamard manifold, then the geodesic through $x^k$ with velocity $-\mbox{grad} f(x^k)$ is minimal and, hence
\begin{equation}\label{eq:minim1}
d(x^{k+1},x^k)=t_k\|\grad f(x^k)\|, \qquad k\in\mathbb{N}.
\end{equation}
Therefore, combining the last equality with first inequality in \eqref{equ:quasidist10} and assumption \eqref{eq:assumpmod1}, the condition \eqref{eq:sdm3} is obtained with $r_2=1/(s_1\tau_{\gamma})$, and the proof is completed.
\end{proof}
\begin{remark}
Note that we could obtain \eqref{eq:minim1} in a more general situation, namely, if the injectivity radius of $M$ is bounded from below by a constant $r > 0$. We recall that the injectivity radius of $M$ (see, for example, do Carmo~\cite[page 271]{MP92} or Sakai~\cite[Definition 4.12, page 110]{S96}) is defined by
$\mbox{\it i}(M):=\inf\{\mbox{\it i}_x: x\in M\},$
where
\[
\mbox{\it i}_x:=\sup\{\epsilon >0:exp_{x}\mid_{B_{\epsilon}(0_{x})} \, \mbox{is a diffeomorphism}\}.
\]
In the particular case that $M$ is a Hadamard manifold $\mbox{\it i}(M):=+\infty$. If the sectional curvature $K$ of $M$ satisfies
\[
0<K_{\min}\leq K\leq K_{\max},
\]
there exists $r>0$ such that $\mbox{\it i}(M)\geq r>0$ (see do Carmo \cite[page 275]{MP92}). However, $i(M)$ may be equal to zero for a complete but noncompact Riemannian manifold $M$, see Sakai~\cite[page 112]{S96}. The assumption on injectivity radius was used by Lageman \cite[Theorem 2.1.19, page 92]{Lageman2007} and Absil et al. \cite[Theorem 7.4.3, page 149]{absil2008}.
\end{remark}

Next, we present the main result of this section.
\begin{theorem}
Assume that $f$ is bounded from below and that there exist $s_1,s_2>0$ such that \eqref{equ:quasidist10} holds. If $f$ is a KL function, then each bounded sequence $\{x^k\}$ generated by \eqref{eq:sdm2} and \eqref{eq:sdm3} converges to some critical point $\bar{x}$ of $f$.
\end{theorem}
\begin{proof}
The assumption {\bf H1} follows from \eqref{eq:sdm2} with $a=r_1$. Note that,
\[
\|\mbox{grad}f(x^{k+1})\|=\|\mbox{grad}f(x^{k+1})-P_{x^k,x^{k+1}}\mbox{grad}f(x^{k})+P_{x^k,x^{k+1}}\mbox{grad}f(x^{k})\|,\qquad k\in\mathbb{N}.
\]
Thus, from the triangle inequality, using that $f$ has $L$-Lipschitz gradient and condition \eqref{eq:sdm3}, it follows that the assumption {\bf H2} holds with $b=L+r_2$. Since $\{x^k\}$ is bounded, by the Roph-Rinow's theorem it has an accumulation point on $M$. Assumption {\bf H3} follows immediately from the definition of $f$. Finally, assumption {\bf H4} follows trivially of Lemma~\ref{lemma:conver1} combined with the condition \eqref{equ:quasidist10} and Hopf-Rinow's Theorem. Therefore, being $f$ a KL function, the result of the theorem follows by directly applying Theorem~\ref{theoconverg}.
\end{proof}
\section{Conclusion}
In this paper we present an unified framework for the convergence analysis of classical descent methods when the objective function satisfies Kurdyka-Lojasiewicz inequality. In particular we answer Problems 1 and 2 presented in the introduction.

\end{document}